\documentclass[10pt]{amsart}
\usepackage{cancel}
\usepackage{amsmath}
\usepackage{mathtools}
\usepackage{amssymb}
\usepackage{amsfonts}
\usepackage{amsthm}
\usepackage{mathrsfs}
\usepackage{mathbbol}
\usepackage{latexsym,amscd}
\usepackage[all]{xy}
\usepackage{color}

\newtheorem{thm}{Theorem}
\newtheorem{prop}{Proposition}
\newtheorem{lem}{Lemma}
\numberwithin{equation}{section}
\numberwithin{prop}{section}
\numberwithin{lem}{section}
\numberwithin{thm}{section}
\newtheorem{cor}{Corollary}
\numberwithin{cor}{section}
\newtheorem{conj}{Conjecture}

\theoremstyle{definition}

\numberwithin{defn}{section}

\newtheorem{rem}{Remark}
\numberwithin{rem}{section}

\newcommand{\ZZ}{\mathbb{Z}}
\newcommand{\CC}{\mathbb{C}}

\newcommand{\be}{\begin {equation}}
\newcommand{\ee}{\end{equation}}

\newcommand{\bee}{\begin {equation*}}
\newcommand{\eee}{\end{equation*}}

\newcommand{\nop}[1]{{}^{\scriptscriptstyle{\circ}}_{\scriptscriptstyle{\circ}}{#1}{}^{\scriptscriptstyle{\circ}}_{\scriptscriptstyle{\circ}}}

\def \a{\alpha }

\def \H{\mathcal{H}}
\def \F{\mathcal{F}}

\def \o{\omega}
\def \op{\omega^{0}}

\def \b{\beta}

\def \H{\mathcal{H}}
\def \1{\mathbb{1}}

\usepackage{tikz-cd}

\begin{document}
\title{Permutation orbifolds of the Heisenberg Vertex Algebra $\H(3)$}  
\author{Antun Milas, Michael Penn, Hanbo Shao}
\address{Department of Mathematics and Statistics, SUNY-Albany}
\email{amilas@albany.edu}
\address{Mathematics Department, Randolph College}
\email{mpenn@randolphcollege.edu}
\address{Department of Mathematics and Computer Science, Colorado College}
\email{h\_shao@coloradocollege.edu}

\maketitle

\begin{abstract} 
We study the $S_3$-orbifold of a rank three Heisenberg vertex algebra 
in terms of generators and relations. 
By using invariant theory, we prove that the orbifold algebra has a minimal strong generating 
set of vectors whose  conformal weights are $1,2,3,4,5,6^2$ (two generators of degree $6$). The structure 
of the cyclic $\mathbb{Z}_3$-oribifold is determined by similar methods.
We also study characters of modules for the orbifold algebra.
\end{abstract}

\section{Introduction}

For every vertex operator algebra $V$, the $n$-fold tensor product $V^{\otimes^n}$ has a natural vertex operator algebra structure.
The symmetric group $S_n$ acts on $V(n):=V^{\otimes n}$ by permuting tensor factors and thus $S_n \subset {\rm Aut}(V(n))$.
Denote by $V(n)^{S_n}$ the fixed point vertex operator subalgebra also called the $S_n$-orbifold of $V(n)$. It is an open 
problem to classify irreducible modules  of $V(n)^{S_n}$ although it is widely believed that every such module should 
come from a $g$-twisted $V^{\otimes n}$-module for some $g \in S_n$.
When it comes to the inner structure  of $V(n)^{S_n}$ (e.g., a strong system of generators) very little is known. 
Even for the Heisenberg orbifold $\H(n)^{S_n}$, where $\mathcal{H}$ is the rank one Heisenberg vertex algebra, this problem seems quite difficult.

Finite and permutation orbifolds have been extensively studied both in physics and mathematics literature. 
According to \cite{B2}, the earliest study of permutation orbifolds seems to be \cite{KS}. 
The first systematic construction of cyclic orbifolds, including their twisted sectors appeared in \cite{BHS}. 
Explicit formulas for characters and modular
transformation properties of permutation orbifolds of Rational Conformal Field Theories were given in \cite{B}. 

In the literature on vertex algebras, the main focus has been on the classification and construction of twisted modules starting with \cite{FLM,Le}.
Further developments include ``Quantum Galois" theory developed in \cite{DM}, work  on twisted sectors of permutation orbifolds   \cite{BDM} (see also \cite{BHL}), tensor category structure for general $G$-orbifolds of rational vertex algebras \cite{DRX1,Ki}, etc.
There are also numerous papers on orbifold vertex algebras for abelian groups and groups of small order.
Structure and representations of $2$-permutation orbifolds were subjects of \cite{A0,A,DN}; see also a more recent work on  $3$-permutation orbifolds of lattice vertex algebras \cite{DRX2}.
The ADE orbifolds of rank one lattice vertex algebras (e.g.  \cite{DJ}) are important for classification of $c=1$ rational vertex algebras.
 Orbifolds of irrational $C_2$-cofinite vertex algebras have been investigated in \cite{ALM}, \cite{CL} and \cite{A}.
 A. Linshaw  (see \cite{CL, L1,L2,L3}, etc.) extensively studied orbifolds of ``free field" vertex operator (super)algebras using the classical invariant theory \cite{W} and its (super) extensions.
As a consequence of the main result in \cite{L3}, every finite orbifold of an affine  vertex algebra is finitely strongly generated.
In particular, this implies that  $\mathcal{H}(n)^{S_n}$ has a finite strong set of generators. Characters of orbifolds of affine vertex algebras were investigated earlier in \cite{KT}.
 
 In this paper we are concerned with the structure of one of the simplest non-abelian orbifolds, $\H(3)^{S_3}$ where $\H$ is the rank one Heisenberg vertex algebra (case $\H(2)^{S_2}$ is well-understood \cite{DN,A0}). We prove three main results.
The first result pertains to generators of $\H(3)^{S_3}$. We show in Theorem 4.1 that this vertex algebra is isomorphic to a $W$-algebra of type $(2,3,4,5,6^2)$ tensored with a rank one Heisenberg vertex algebra. Here labels $2,3,4,5,6^2$ indicate that our $W$-algebra is strongly generated by the Virasoro vector (of degree two) and five primary vectors of degrees $3,4,5$ and two of degree $6$. These five generators are explicitly given in Section 4 where we denoted them by $J_1,J_2,C_1,C_2$ and $C_3$.
Our second result is about the cyclic $\mathbb{Z}_3$-orbifold of $\mathcal{H}(3)$ (see Theorem  5.1 for details).

In the last part (Sections 6 and 7) we discuss  characters of certain $\H(3)^{S_3}$-modules and their modular properties. It is expected that many irrational vertex algebras will enjoy modular invariance  in a generalized sense involving iterated integrals instead of sums. For the rank $n$ permutation orbifold $\H(n)^{S_n},$ the character of a module $M$ is expected to transform as
 \begin{equation} \label{gen-mod}
 {\rm ch}[M] \left(-\frac{1}{\tau}\right)=\sum_{i \geq 1}^n \int_{\mathbb{R}^i} S_{M, M_{\lambda_i}} {\rm ch}[M_{\lambda_i}](\tau) d \lambda_i,
 \end{equation}
 where $S_{M,M_{\lambda_i}} \in \mathbb{C}$ and $\lambda_i \in \mathbb{R}^i$, $1 \leq i \leq n$ parametrize certain $\H(n)^{S_n}$-modules. Our third main result gives strong evidence for this conjecture
for $n=3$  (see Theorem 7.1).

 




\section{The $S_n$-orbifold of the Heisenberg vertex operator algebra}

Let $\H$ denote the rank one Heisenberg vertex operator algebra generated by $\alpha(-1)\1$, with the usual conformal vector (and grading) given by $\omega=\frac{1}{2} \alpha^2(-1)\1$.
Let $\H(n)=\H^{\otimes^n}$. 
For convenience, we suppress the tensor product symbol and let  $\a_1(-1)\1:=\a(-1)\1 \otimes  \cdots  \otimes \1 \in \H(n)$, and similarly we define $\a_i(-1)\1$, $i \geq 2$, 
 such that
$\H(n)=\langle \alpha_1(-1)\1, \cdots ,\alpha_n(-1)\1 \rangle$.
We consider the natural action of $S_n$ on $\H(n)$ given by 
\be
\sigma\cdot \a_{i_1}(m_1)\cdots \a_{i_k}(m_k)\1=\a_{\sigma(i_1)}(m_1)\cdots \a_{\sigma(i_k)}(m_k)\1,
\ee
for $1\leq i_j\leq n$, $m_j<0$, and $\sigma\in S_n$. 


We clearly have a natural linear isomorphism 
\be\label{linearisom}
\H(n)\cong \CC[x_i(m) | 1\leq i \leq n, m\geq 0]\ee
induced by $\a_i(-m-1)\mapsto x_i(m)$ for $m\geq 0$. Using the terminology of \cite{L2}, we say that $\CC[x_i(m) | 1\leq i \leq n, m\geq 0]$ is the associated graded algebra of the vertex algebra $\H(n)$. Further, we may endow the polynomial algebra $\CC[x_i(m) | 1\leq i \leq n, m\geq 0]$ with the structure of a $\partial$-ring by defining the map
\be\begin{aligned}
\partial:\CC[x_i(m) | 1\leq i \leq n, m\geq 0]&\to \CC[x_i(m) | 1\leq i \leq n, m\geq 0]\\
x_i(m)&\mapsto (m+1)x_i(m+1),\end{aligned}\ee
where the action of $\partial$ is extended to the whole space via the Leibniz rule. 
This definition of $\partial$  is compatible with the translation operator in $\H(n)$ given by $T(v)=v_{-2}\mathbb{1}$.

The following Lemma is from  \cite{L1}.
\begin{lem}\label{reconstruction}
Let $\mathcal{A}$ be a vertex algebra with a $\mathbb{Z}_{\geq 0}$ filtration, where $\tilde{\mathcal{A}}$ is the associated $\partial$-ring. If $\{\tilde{a}_i|i\in I\}$ generates $\tilde{\mathcal{A}}$ then $\{a_i | i\in I\}$ strongly generates $\mathcal{A}$, where $a_i$ and $\tilde{a}_i$ are related via the natural linear isomorphism described by the $\mathbb{Z}_{\geq 0}$ filtration.
\end{lem}

Further, we recall that the invariant ring $\CC[x_1,\dots,x_n]^{S_n}$ has a variety of generating sets, including the power sum polynomials 
\begin{align*}
& p_i=x_1^i+\cdots+x_n^i, \ \ 1\leq i\leq n.
\end{align*}
In addition to this, it is common to study the invariant theory of the ring of infinitely many commuting copies of this polynomial algebra, where we denote by $x_i(m)$ the copy of $x_i$ from the $m^{\text{th}}$ copy of the polynomial algebra.  A Theorem of Weyl \cite{W} shows that $\CC[x_i(m) | 1\leq i \leq n, m\geq 0]^{S_n}$ is generated by the polarizations of these polynomials
\be\label{dringgenerators}
q_k(m_1,\dots,m_k)=\sum_{i=1}^n x_i(m_1)\cdots x_i(m_k),\ee
for $1\leq k\leq n$.
Now, applying Lemma \ref{reconstruction}, we have an initial strong generating set for the orbifold $\H(n)^{S_n}$ given by the vectors
\be\label{firstorbifoldgenerators}
\omega_k(m_1,\dots,m_k)=\sum_{i=1}^n \a_i(-1-m_1)\cdots \a_i(-1-m_k)\1,
\ee
for $1\leq k\leq n$ and $m_j\geq 0$. It should be noted that the conformal vector of $\H(n)$ is $\frac{1}{2}\o_2(0,0)$, making the orbifold a vertex operator algebra.

\section{Warmup: $\H(2)^{S_2}$}


Here we describe the structure of $\H(2)^{S_2}$. This case is well-known, so we only provide a few details. 
Denote by $M(1)^+:=\H(1)^{\ZZ_2} $ the fixed point subalgebra under the action $\a(-1)\1 \to -\a(-1)\1$, where $\a$ is the Heisenberg generator (in the physics literature $M(1)^+$ is often denoted by 
$W(2,4)$).
This vertex algebra was thoroughly studied in \cite{DN} and elsewhere.  
Let
\begin{align*}
h& =\a_1(-1)\1+\a_2(-1)\1, \ \ \ h^\perp=\a_1(-1)\1-\a_2(-1)\1,  \\ 
& \ \ \omega=\frac14 \a_1^2(-1)\1+\frac14 \a_2^2(-1)\1- \frac12 \a_1(-1) \a_2(-1)\1.
\end{align*}
Observe that as vertex algebras 
$$\H(2)=\H(1)_{h} \otimes \H(1)_{h^\perp}$$
where $\H(1)_{h}=\langle h \rangle$ and $\H(1)_{h^\perp}=\langle h^\perp \rangle$ with conformal vector $\omega$.
The nontrivial element of the group $S_2$ fixes the first tensor factor and 
$$h^\perp \to - h^{\perp}.$$ Thus we immediately get 
\begin{equation} \label{iso}
\H(2)^{S_2} \cong \H(1) \otimes M(1)^+ .
\end{equation}
%
It is easy to see that
$$H=\frac{1}{2} \biggl( \a_1(-1)^4 +\a_2(-1)^4- 4 \a_1^3(-1) \a_2(-1) -4 \a_1(-1)\a_2^3(-1)+6 \a_1^2(-1) \a_2^2(-1)+3 \a_1^2(-2)+3 \a_2^2(-2)$$
$$- 6 \a_1(-2)\a_2(-2)+2 \a_1(-1)\a_2(-3)+2\a_1(-3)\a_2(-1)-2 \a_1(-3)\a_1(-1)-2\a_2(-3)\a_2(-1) \biggr)\1$$
is contained in $ \langle h^\perp \rangle$ and is  primary  of conformal weight $4$, and thus a generator of  $M(1)^+$.





\section{The Rank Three Case}

\subsection{The orbifold $\H(3)^{S_3}$}

As described above we know that the orbifold $\H(3)^{S_3}$ will be strongly generated by the vectors
\be\label{initialgenerators}\begin{aligned}
\omega_1(a)&=\sum_{i=1}^3 \a_i(-1-a)\1, \\
 \omega_2(a,b)&=\sum_{i=1}^3\a_i(-1-a)\a_i(-1-b)\1, \\ 
\omega_3(a,b,c)&=\sum_{i=1}^3\a_i(-1-a)\a_i(-1-b)\a_i(-1-c)\1.
\end{aligned}\ee

The following change of basis of the generating set will allow for an efficient reduction in the generating set of the corresponding orbifold
\be\begin{aligned}
\b_1(-1)\1&=\frac{1}{\sqrt{3}}(\a_1(-1)\1+\a_2(-1)\1+\a_3(-1)\1),\\
\b_2(-1)\1&=\frac{1}{\sqrt{3}}(\a_1(-1)\1+\eta^2\a_2(-1)\1+\eta\a_3(-1)\1),\\
\b_3(-1)\1&=\frac{1}{\sqrt{3}}(\a_1(-1)\1+\eta\a_2(-1)\1+\eta^2\a_3(-1)\1),\\
\end{aligned}\ee
where $\eta$ is a primitive third root of unity. Using this generating set we have $\sigma\cdot \beta_1(-1)\1=\beta_1(-1)\1$ for all $\sigma\in S_3$. Further, examining the action of the generators of $S_3$ on the generators of $\H(3)$ we have
\be\begin{aligned}
\begin{pmatrix}2 &3 \end{pmatrix}\cdot\b_2(-1)\1&=\b_3(-1)\1, \ \ \ \ \  \ \ \begin{pmatrix}2 & 3 \end{pmatrix}\cdot\b_3(-1)\1=\b_2(-1)\1,\end{aligned}\ee
and
\be\begin{aligned}
\begin{pmatrix}1&2&3\end{pmatrix}\cdot\b_2(-1)\1&=\eta \b_2(-1)\1, \ \ \ \ \  \ \ \begin{pmatrix}1&2&3\end{pmatrix}\cdot\b_3(-1)\1=\eta^2\b_3(-1)\1.\end{aligned}\ee
From this action, we see that an initial generating set for the orbifold $\H(3)^{S_3}$ may be taken to be
\be\begin{aligned}\label{initialgenerators2}
\op_1(a)&=\b_1(-1-a)\1,\\
\op_2(a,b)&=\b_2(-1-a)\b_3(-1-b)\1+\b_3(-1-a)\b_2(-1-b)\1,\\
\op_3(a,b,c)&=\b_2(-1-a)\b_2(-1-b)\b_2(-1-c)\1+\b_3(-1-a)\b_3(-1-b)\b_3(-1-c)\1,
\end{aligned}\ee
for $0\leq a\leq b\leq c$. In fact, we can explicitly write the relation between our original generators (\ref{initialgenerators}) and our new generators (\ref{initialgenerators2}) as follows
\be\begin{aligned}\label{generatortranslation}
\o_1(a)&=\sqrt{3}\op_1(a),\\
\o_2(a,b)&=\op_2(a,b)+\op_1(a)_{-1}\op_1(b),\\
\o_3(a,b,c)&=\frac{1}{3}\op_3(a,b,c)+\frac{1}{\sqrt{3}}\left(\op_1(a)_{-1}\op_2(b,c)+\op_1(b)_{-1}\op_2(a,c)\right.\\& \hspace{0.8in}\left.+\op_1(c)_{-1}\op_2(a,b)+\op_1(a)_{-1}\op_1(b)_{-1}\op_1(c)\right).
\end{aligned}\ee

We may make a similar change of variables for the generating set of $\CC[x_i(m)|1\leq i\leq 3, m\geq 0]$ by setting 
\be\begin{aligned}
y_1(m_1)&=\frac{1}{\sqrt{3}}(x_1(m_1)+x_2(m_1)+x_3(m_1)),\\
y_2(m_2)&=\frac{1}{\sqrt{3}}(x_1(m_2)+\eta^2x_2(m_2)+\eta x_3(m_2)),\\
y_3(m_3)&=\frac{1}{\sqrt{3}}(x_1(m_3)+\eta x_2(m_3)+\eta x_3(m_3)),\end{aligned}\ee
for $m_i\geq 0$, and 
\be\begin{aligned}\label{dringgenerators2}
q_1^0(a)&=y_1(a),\\
q_2^0(a,b)&=y_2(a)y_3(b)+y_3(a)y_2(b),\\
q_3^0(a,b,c)&=y_2(a)y_2(b)y_2(c)+y_3(a)y_3(b)y_3(c),\end{aligned}\ee
for $a,b,c\geq 0$.

Using the translation operator, $T$, restricted to the orbifold $\mathcal{H}(3)^{S_3}$, the initial strong generating set (\ref{initialgenerators2}) can be reduced per the following Lemma.

\begin{lem}\label{orbifoldgenerators1} The orbifold $\mathcal{H}(3)^{S_3}$ is strongly generated by the vectors
\be\begin{aligned}
\o_1^0(0),&\\
\o_2^0(0,2a)& \text{ for } a\geq 0,\\
\o_3^0(0,a,b)& \text{ for } 0\leq a \leq b.\end{aligned}\ee
\end{lem}
\begin{proof}

 It is clear that since
\be
\op_1(a)_{-1}\1=\frac{1}{a!}\op_1(0)_{-1-a}\1
\ee
the linear portion of the generating set (\ref{initialgenerators2}) can be immediately minimized to contain the single weight one element $\omega_1(0)$. 

Moving on toward the quadratic terms in (\ref{initialgenerators2}), we consider the vector spaces
\be
A_2(m)=\text{span }\{\op_2(a,b)|a+b=m\},\ee
and a natural family of subspaces
\be
\partial A_2(m)=\{v_{-2}\1|v\in A_2(m)\}.\ee
Observe that 
\be\begin{aligned}\label{derivative}
\op_2(a,b)_{-2}\1&=\text{Res }z^{-1}\frac{\partial}{\partial z}\left(Y(\b_2(-1-a)\b_3(-1-b)\1,z)+Y(\b_3(-1-a)\b_2(-1-b)\1,z)\right)\\
&=\frac{1}{a!b!}\text{Res }z^{-1}\left(\nop{\frac{\partial^a}{\partial z^a}\b_2(z)\frac{\partial^b}{\partial z^b}\b_3(z)}+\nop{\frac{\partial^a}{\partial z^a}\b_3(z)\frac{\partial^b}{\partial z^b}\b_2(z)}\right)\\&=(a+1)\op_2(a+1,b)+(b+1)\op_2(a,b+1),\end{aligned}\ee
where $\b_i(z)=Y(\b_i(-1)\1,z)$. Also notice that for all $a,b\geq 0$ we have $\op_2(a,b)=\op_2(b,a)$. We may take an initial basis for $A_2(m)$ to be the set
\be
\mathcal{B}_2(m)=\left\{\op_2(i,m-i)|0\leq i\leq \left\lfloor\frac{m}{2}\right\rfloor\right\}\ee 
which implies that 
\be
\text{dim } A_2(2n+1)=\text{dim }A_2(2n)=n+1.\ee
Further the set 
\be
\{v_{-2}\1|v\in \mathcal{B}_2(m)\}\ee
is clearly a basis for $\partial A_2(m)$. It follows that 
\be\begin{aligned}
A_2(2n)&=\partial A_2(2n-1)\oplus \CC\op_2(0,2n),\\
A_2(2n+1)&=\partial A_2(2n)\end{aligned}\ee
and thus by induction we may take a more convenient basis of $A_2(m)$ to be 
\be
\left\{\op_2(0,2i)_{-1-m+2i}\1|0\leq i\leq \left\lfloor\frac{m}{2}\right\rfloor\right\}.\ee 
Thus, a more efficient set of quadratic generators for the $\H(3)^{S_3}$ may be taken to be $\op_2(0,2a)$ for $a\geq 0$.

Finally, we consider the cubic generators.  Throughout, we use the fact that $\o_3^0(a,b,c)$ is invariant under any permutation of the entries $a,b,c$.  Analogous to \eqref{derivative} we have 
\be
\o_3^0(a,b,c)_{-2}\1=(a+1)\o_3^0(a+1,b,c)+(b+1)\o_3^0(a,b+1,c)+(c+1)\o_3^0(a,b,c+1),
\ee
which may be used inductively to write all cubic generators in terms of those of the form $\o_3^0(0,a,b)$ with $0\leq a\leq b$.

\end{proof}

\begin{rem}Using (\ref{generatortranslation}) we may take 
\be\begin{aligned}\label{naturalgenerators1}
\o_1(0),&\\
\o_2(0,2a)& \text{ for } a\geq 0,\\
\o_3(0,a,b)& \text{ for } 0\leq a \leq b,\end{aligned}\ee
as our strong generating set. We take advantage of this translation tool as generators of this form are somewhat more natural to the parent algebra, $\H(3)$. 
\end{rem}
\begin{rem} 
This simplification of the generating set also holds if we consider the $\partial$-ring associated to $\H(3)^{S_3}$, which has un-reduced generators given by (\ref{dringgenerators}). We may reduce these to a minimal generating set for $\CC[x_i(m)|1\leq i\leq 3, m\geq 0]^{S_3}$ given by 
\be\begin{aligned}\label{dringgenerators1}
q_1(0),&\\
q_2(0,2a)& \text{ for } a\geq 0,\\
q_3(0,a,b)& \text{ for } 0\leq a \leq b.\end{aligned}\ee

\end{rem}

We now present the following relations among the generators of $\CC[x_i(m)|1\leq i\leq 3, m\geq 0]^{S_3}$.

\begin{lem}\label{classicalrelation}
For $\mathbf{a}=(a_1,a_2,a_3,a_4,a_5,a_6)$ with $a_1,a_2,a_3,a_4,a_5,a_6\geq 0$ we have 
\be\begin{aligned}
D^{C,1}_6(\mathbf{a})&=q_2^0(a_1,a_2)q_2^0(a_3,a_4)q_2^0(a_5,a_6)-q_2^0(a_1,a_2)q_2^0(a_3,a_6)q_2^0(a_4,a_5)\\&+q_2^0(a_1,a_4)q_2^0(a_2,a_6)q_2^0(a_3,a_5)-q_2^0(a_1,a_4)q_2^0(a_2,a_3)q_2^0(a_5,a_6)\\&+q_2^0(a_1,a_5)q_2^0(a_2,a_4)q_2^0(a_3,a_6)-q_2^0(a_1,a_5)q_2^0(a_2,a_6)q_2^0(a_3,a_4)\\&+q_2^0(a_1,a_6)q_2^0(a_2,a_3)q_2^0(a_4,a_5)-q_2^0(a_1,a_6)q_2^0(a_2,a_4)q_2^0(a_3,a_5)=0,\end{aligned}\ee

\be\begin{aligned}
D^{C,2}_6(\mathbf{a})&=q_3^0(a_1,a_2,a_3)q_3^0(a_4,a_5,a_6)-q_3^0(a_1,a_2,a_4)q_3^0(a_3,a_5,a_6)\\
&+\frac{1}{2}q_2^0(a_1,a_3)q_2^0(a_2,a_6)q_2^0(a_4,a_5)-\frac{1}{2}q_2^0(a_1,a_4)q_2^0(a_2,a_5)q_2^0(a_3,a_6)\\
&+\frac{1}{2}q_2^0(a_1,a_5)q_2^0(a_2,a_3)q_2^0(a_4,a_6)-\frac{1}{2}q_2^0(a_1,a_5)q_2^0(a_2,a_4)q_2^0(a_3,a_6)=0,
\end{aligned}\ee
and
\be\begin{aligned}
D_5^C(\mathbf{a})&=q_2^0(a_1,a_2)q_3^0(a_3,a_4,a_5)-q_2^0(a_1,a_5)q_3^0(a_2,a_3,a_4)\\&-q_2^0(a_2,a_5)q_3^0(a_1,a_3,a_4)-q_2^0(a_3,a_4)q_3^0(a_1,a_2,a_5)\\&+q_2^0(a_3,a_5)q_3^0(a_1,a_2,a_4)+q_2^0(a_4,a_5)q_3^0(a_1,a_2,a_3)=0,\end{aligned}\ee
for $\mathbf{a}=(a_1,a_2,a_3,a_4,a_5)$ with $a_1,a_2,a_3,a_4,a_5\geq0$.
\end{lem}

\begin{rem} These relations play the role of the determinant (and similar) relations in the invariant theory of  classical Lie groups. We do not claim that these two families of expressions generate all relations in this case. In fact, there are such relations at every degree, and since the Frobenius number of $5$ and $6$ is $19$, there are new relations at least up to this degree.  \end{rem}

We identify the relations from Lemma \ref{classicalrelation} via the isomorphism \eqref{linearisom} with expressions involving the generators of the orbifold $\mathcal{H}(3)^{S_3}$. For $\mathbf{a}=(a_1,a_2,a_3,a_4,a_5,a_6)$ with $a_1,a_2,a_3,a_4,a_5,a_6\geq 0$ set 
\be\begin{aligned}
D^1_6(\mathbf{a})&=\o_2^0(a_1,a_2)_{-1}\o_2^0(a_3,a_4)_{-1}\o_2^0(a_5,a_6)-\o_2^0(a_1,a_2)_{-1}\o_2^0(a_3,a_6)_{-1}\o_2^0(a_4,a_5)\\&+\o_2^0(a_1,a_4)_{-1}\o_2^0(a_2,a_6)_{-1}\o_2^0(a_3,a_5)-\o_2^0(a_1,a_4)_{-1}\o_2^0(a_2,a_3)_{-1}\o_2^0(a_5,a_6)\\&+\o_2^0(a_1,a_5)_{-1}\o_2^0(a_2,a_4)_{-1}\o_2^0(a_3,a_6)-\o_2^0(a_1,a_5)_{-1}\o_2^0(a_2,a_6)_{-1}\o_2^0(a_3,a_4)\\&+\o_2^0(a_1,a_6)_{-1}\o_2^0(a_2,a_3)_{-1}\o_2^0(a_4,a_5)-\o_2^0(a_1,a_6)_{-1}\o_2^0(a_2,a_4)_{-1}\o_2^0(a_3,a_5),\end{aligned}\ee

\be\begin{aligned}
D^{2}_6(\mathbf{a})&=\o_3^0(a_1,a_2,a_3)_{-1}\o_3^0(a_4,a_5,a_6)-\o_3^0(a_1,a_2,a_4)_{-1}\o_3^0(a_3,a_5,a_6)\\
&+\frac{1}{2}\o_2^0(a_1,a_3)_{-1}\o_2^0(a_2,a_6)_{-1}\o_2^0(a_4,a_5)-\frac{1}{2}\o_2^0(a_1,a_4)_{-1}\o_2^0(a_2,a_5)_{-1}\o_2^0(a_3,a_6)\\
&+\frac{1}{2}\o_2^0(a_1,a_5)_{-1}\o_2^0(a_2,a_3)_{-1}\o_2^0(a_4,a_6)-\frac{1}{2}\o_2^0(a_1,a_5)_{-1}\o_2^0(a_2,a_4)_{-1}\o_2^0(a_3,a_6),
\end{aligned}\ee
and
\be\begin{aligned}
D_5(\mathbf{a})&=\o_2^0(a_1,a_2)_{-1}\o_3^0(a_3,a_4,a_5)-\o_2^0(a_1,a_5)_{-1}\o_3^0(a_2,a_3,a_4)\\&-\o_2^0(a_2,a_5)_{-1}\o_3^0(a_1,a_3,a_4)-\o_2^0(a_3,a_4)_{-1}\o_3^0(a_1,a_2,a_5)\\&+\o_2^0(a_3,a_5)_{-1}\o_3^0(a_1,a_2,a_4)+\o_2^0(a_4,a_5)_{-1}\o_3^0(a_1,a_2,a_3),\end{aligned},\ee
for $\mathbf{a}=(a_1,a_2,a_3,a_4,a_5)$ with $a_1,a_2,a_3,a_4,a_5\geq 0$.

 Observe that for $i\in\{1,2\}$, the weight of the expression $D^i_6(\mathbf{a})$ is $|\mathbf{a}|+6$ while the weight of $D_5(\mathbf{a})$ is $|\mathbf{a}|+5$, where we take $|(u_1,\dots,u_n)|=u_1+\cdots+u_n$ for any multi-index. By Lemma \ref{classicalrelation} and repeated applications of the weak associativity properties of vertex algebras, along with our linear isomorphism (\ref{linearisom}) we see that for $i\in\{1,2\}$, we may rewrite
\be\label{quaddecomp1}
D^i_6(\mathbf{a})=D_6^{(4)}(\mathbf{a})+D_6^{(2)}(\mathbf{a}),\ee
where 
\be\label{quaddecomp2}
D_6^{(4)}(\mathbf{a})=\sum_{\substack{a,b,c,d\geq0\\a+b+c+d=|\mathbf{a}|+2}}\mu^{(4)}_{a,b,c,d}\o^0_2(a,b)_{-1}\o^0_2(c,d)\ee
and
\be\label{quaddecomp3}
D_6^{(2)}(\mathbf{a})=\sum_{\substack{a,b\geq 0\\a+b=|\mathbf{a}|+4}}\mu^{(2)}_{a,b}\o_2^0(a,b),\ee
where $\mu^{(4)}_{a,b,c,d}$ and $\mu^{(2)}_{a,b}$ are appropriate constants. Importantly, the right hand side of \eqref{quaddecomp2} has no terms that are ``cubic'' in the generators $\o^0_2(a,b)$ or ``quadratic'' in the generators $\o^0_3(a,b,c)$, each of which would contain a combination of six of the original $\beta_i$ vectors with $i\in\{2,3\}$. We have a similar (and simpler) decomposition 
\be\label{cubicdecomp1}
D_5(\mathbf{a})=D^{(3)}_5(\mathbf{a}),\ee
where  
\be \label{cubicdecomp2}
D^{(3)}_5(\mathbf{a})=\sum_{\substack{a,b,c\geq 0 \\ a+b+c=|\mathbf{a}|+2}}\mu_{a,b,c}^{(3)}\o_3^0(a,b,c),\ee
where $\mu_{a,b,c}^{(3)}$ are constants. Again, Lemma \ref{classicalrelation}, weak associativity, and (\ref{linearisom}) implies that the right hand side of \eqref{cubicdecomp1} will not contain terms that are ``products'' of the generators $\o^0_2(a,b)$ and $\o^0_3(a,b,c)$ or otherwise a combination of five of the original vectors $\beta_2$ and $\beta_3$.

Now we are poised to use the expressions $D^i_6(\mathbf{a})$ and $D_5(\mathbf{a})$ to further reduce the generating set of $\mathcal{H}(3)^{S_3}$ described in Lemma \ref{orbifoldgenerators1}. 

\begin{lem} \label{quad-lem} The full list of quadratic generators described in Lemma \ref{orbifoldgenerators1} can be replaced with the set 
$\{\o_2^0(0,0),\o_2^0(0,2),\o_2^0(0,4)\}$.\end{lem}

\begin{proof} 

Theorem 4.7 of \cite{L2} implies that, \color{black} using only combination of quadratic generators, \color{black} we may immediately reduce our quadratic generators to the set 
\be\label{andyset}\{\o_2^0(0,0),\o_2^0(0,2),\o_2^0(0,4), \o_2^0(0,6), \o_2^0(0,8)\}\ee

\color{black}In order to remove the remaining two quadratic generators we use decoupling relations that involve both quadratic and cubic generators. \color{black} Using the expression $D_6^1(0,0,0,0,1,1)$, the decomposition described in (\ref{quaddecomp1}-\ref{quaddecomp3}) may be used to construct the following equation
\be\begin{aligned}\label{quaddec1}
 \o_2^0(0,6)&=\frac{143}{742}\o_2^0(0,4)_{-3}\1-\frac{81}{371}\o_2^0(0,2)_{-5}\1+\frac{743}{1113}\o_2^0(0,0)_{-7}\1 
 +\frac{1}{1484}\o_2^0(0,0)_{-1}\o_2^0(0,4)\\&-\frac{13}{2968}\o_2^0(0,0)_{-1}\o_2^0(1,3)-\frac{1}{4452}\o_2^0(0,0)_{-1}\o_2^0(2,2)+\frac{2}{1113}\o_2^0(0,1)_{-1}\o_2^0(0,3)\\&-\frac{1}{1113}\o_2^0(0,1)_{-1}\o_2^0(1,2)+\frac{1}{2226}\o_2^0(0,2)_{-1}\o_2^0(0,2)-\frac{3}{1484}\o_2^0(0,2)_{-1}\o_2^0(1,1)\\&-\frac{1}{2968}\o_2^0(1,1)_{-1}\o_2^0(1,1)+\frac{1}{4452}\o_2^0(0,0)_{-1}\o_2^0(0,0)_{-1}\o_2^0(0,1)\\&-\frac{1}{8904}\o_2^0(0,0)_{-1}\o_2^0(0,1)_{-1}\o_2^0(0,1)+\frac{1}{4452}\o_3^0(0,0,0)_{-1}\o_3^0(0,1,1)\\&-\frac{1}{4452}\o_3^0(0,0,1)_{-1}\o_3^0(0,0,1),
\end{aligned}\ee
and thus we may remove $\o_2^0(0,6)$ from \eqref{andyset}. Using \eqref{quaddec1} and $D_6^1(0,0,0,0,2,2)$ we can similarly write $\o_2^0(0,8)$ as a vertex algebraic polynomial using the quadratic generators $\o_2^0(0,0)$, $\o^0_2(0,2)$, $\o_2^0(0,4)$, \color{black} along with cubic generators $\o_3^0(0,0,0)$, $\o_3^0(0,0,1)$, $\o_3^0(0,1,1)$, $\o_3^0(0,0,2)$, and $\o_3^0(0,2,2)$. \color{black}

\end{proof}

\begin{lem} \label{cubic-lem} The full list of cubic generators described in Lemma \ref{orbifoldgenerators1} can be be replaced with the set 
\be\label{finalcubicgens}\{\o^0_{3}(0,0,0),\o^0_3(0,0,2),\o_3^0(0,1,2)\}.\ee
\color{black} Moreover, these generators together with $\omega^0_1(0), \o_2^0(0,0),\o_2^0(0,2),\o_2^0(0,4) $ form a strong generating set of $\H(3)^{S_3}$.\color{black}

\end{lem}

\begin{proof} 
To get started we notice that 
\be\begin{aligned}\label{setupderivative}
\o_3^0(0,0,1)&=\frac{1}{3}\o_3^0(0,0,0)_{-2}\1\\
\o_3^0(0,1,1)&=-\o_3^0(0,0,2)+\frac{2}{3}\o_3^0(0,0,0)_{-3}\1\\
\o_3^0(0,0,3)&=-\frac{2}{3}\o_3^0(0,1,2)+\frac{1}{3}\o_3^0(0,0,2)_{-2}\1,\\
\end{aligned}\ee
in other words, all cubic generators of conformal weight 6 or less may be written in terms of the generators$\{\o^0_{3}(0,0,0),\o^0_3(0,0,2),\o_3^0(0,1,2)\}$ using only the translation operator. Furthermore, we have

\be
\begin{aligned}\label{bigderivative}
\o_3^0(0,1,3)&=-\frac{1}{12}\o_3^0(0,0,4)-\frac{1}{72}\o_3^0(0,1,2)_{-2}\1+\frac{1}{72}\o_3^0(0,0,2)_{-3}\1,\\
\color{black}
\o_3^0(0,2,2)&=\color{black}3\o_3^0(0,0,4)+\frac{4}{3}\o_3^0(0,1,2)_{-2}\1-\frac{1}{3}\o_3^0(0,0,2)_{-3}\1-\frac{1}{3}\o_3^0(0,0,0)_{-5}\1,\color{black}\\
\o_3^0(0,2,3)&=-\frac{1}{30}\o_3^0(0,1,4)+\frac{1}{60}\o_3^0(0,0,4)_{-2}\1+\frac{1}{35}\o_3^0(0,1,2)_{-3}\1-\frac{1}{36}\o_3^0(0,0,0)_{-6}\1,\\
\o_3^0(0,3,3)&=\frac{5}{4}\o_3^0(0,0,6)-\frac{1}{2}\o_3^0(0,2,4)+\frac{1}{5}\o_3^0(0,1,4)_{-2}\1+\frac{1}{20}\o_3^0(0,0,4)_{-3}\1\\&+\frac{1}{6}\o_3^0(0,1,2)_{-4}\1-\frac{1}{12}\o_3^0(0,0,2)_{-5}\1-\frac{1}{4}\o_3^0(0,0,0)_{-7}\1,\\
\o_3^0(0,3,4)&=\frac{5}{7}\o_3^0(0,1,6)+\frac{15}{7}\o_3^0(0,0,6)_{-2}\1-\frac{1}{3}\o_3^0(0,2,4)_{-2}\1+\frac{2}{3}\o_3^0(0,1,4)_{-3}\1\\&+\frac{5}{9}\o_3^0(0,1,2)_{-5}\1-\frac{5}{9}\o_3^0(0,0,2)_{-6}-\frac{10}{9}\o_3^0(0,0,0)_{-8}\1,\\
\o_3^0(0,4,4)&=-21\o_3^0(0,0,8)+4\o_3^0(0,2,6)-\frac{8}{7}\o_3^0(0,1,6)_{-2}\1+\frac{29}{7}\o_3^0(0,0,6)_{-3}\1\\&-\frac{2}{3}\o_3^0(0,2,4)_{-3}\1+\frac{8}{5}\o_3^0(0,1,4)_{-4}\1-\frac{1}{5}\o_3^0(0,0,4)_{-5}\1\\&+\frac{16}{9}\o_3^0(0,1,2)_{-6}\1-\frac{7}{3}\o_3^0(0,0,2)_{-7}\1-\frac{47}{9}\o_3^0(0,0,0)_{-9}\1,
\end{aligned}\ee
that is, all vectors of the form $\o_3^0(0,a,b)$ with $0\leq a\leq b\leq 4$ can be written, using only the translation operator, in terms of our proposed generating set, \eqref{finalcubicgens}, with the addition of the vectors $\o_3^0(0,0,4)$, $\o_3^0(0,1,4)$, $\o_3^0(0,2,4)$, and $\o_3^0(0,c,d)$ with $0\leq c\leq d$ where $d\geq 5$. The remainder of our argument will be concerned with eliminating the need for these additional vectors.

Using $D_5(0,0,0,1,1)$ and the decomposition described in (\ref{cubicdecomp1}-\ref{cubicdecomp2}) we have 
\be\begin{aligned}\label{step1}
\o_3^0(0,0,4)&=-\frac{16}{15}\o_3^0(0,1,2)_{-2}\1+\frac{4}{15}\o_3^0(0,0,2)_{-3}\1 +\frac{24}{45}\o_3^0(0,0,0)_{-5}\1-\frac{2}{5}\o_2^0(0,0)_{-1}\o_3^0(0,1,1)\\
&+\frac{4}{5}\o_2^0(0,1)_{-1}\o_3^0(0,0,1)-\frac{2}{5}\o_2^0(1,1)_{-1}\o_3^0(0,0,0),\end{aligned}\ee
where again this calculation was performed using  \cite{T}. Using $D_5(0,0,1,1,1)$ a similar decoupling equation for $\o_3^0(0,1,4)$ can be found. Furthermore a linear combination of $D_5(0,0,0,1,3)$ and $D_5(0,0,1,1,2)$ leads to the decoupling equation
\be\begin{aligned}\label{step2}
\o_3^0(0,2,4)&=\frac{1}{18}\o_2^0(0,0)_{-1}\o_3^0(1,2,2) + \frac{1}{64}\o_2^0(0,0)_{-1}\o_3^0(0,1,3) - \frac{1}{18}\o_2^0(0,1)_{-1}\o_3^0(0,1,2) \\&- 
 \frac{1}{64}\o_2^0(0,1)_{-1}\o_3^0(0,0,3) - \frac{1}{18}\o_2^0(0,2)_{-1}\o_3^0(0,1,1)+ \frac{1}{18}\o_2^0(1,2)_{-1}\o_3^0(0,0,1) \\&- \frac{1}{64}\o_2^0(0,3)_{-1}\o_3^0(0,1,1) + \frac{1}{64}\o_2^0(1,3)_{-1}\o_3^0(0,0,1) -\frac{7}{18}\o_3^0(0,0,0)_{-7}\1\\& - \frac{59}{144}\o_3^0(0,0,2)_{-5}\1 + 
 \frac{187}{240}\o_3^0(0,0,4)_{-3}\1+\frac{115}{288}\o_3^0(0,1,2)_{-4}\1+\frac{37}{60}\o_3^0(0,1,4)_{-2}\1.
\end{aligned}\ee

Next, for $a\in\mathbb{N}$, each of the six expressions $D_5(0,0,0,1,a-3)$, $D_5(0,0,0,2,a-4)$, $D_5(0,0,1,1,a-4)$, $D_5(0,0,0,3,a-5)$, $D_5(0,0,1,2,a-5)$, and $D_5(0,1,1,1,a-5)$ can be expanded as linear combinations of the six vectors $\o_3^0(0,0,a)$, $\o_3^0(0,1,a-1)$, $\o_3^0(0,2,a-2)$, $\o_3^0(0,3,a-3)$, $\o_3^0(0,4,a-4)$, and $\o_3^0(0,5,a-5)$ along with terms of the form $\o_3(0,m_1,m_2)_{-1-n}\1$ where $m_1+m_2+n=a$ with $n\geq 1$. The linear independence of these can be determined by considering the determinant of the matrix, $A$, whose $(i,j)$ entry is the coefficient of the $i^{\text{th}}$ vector in the $j^{\text{th}}$ expression as described above. A lengthy elementary calculation shows that, in the case that $a$ is even
\be\begin{aligned}\label{step3}
\text{det }A&=\frac{1}{6} (a-4)^2 (a-3) (a-1) (a+2) (5331 a^6-70325 a^5+314669 a^4-613567 a^3\\&+97384 a^2+1614156 a-1835568),\end{aligned}\ee
and a similar expression exists for odd $a$. In particular, for $a\geq 5$, we can write each of the vectors $\o_3^0(0,0,a)$, $\o_3^0(0,1,a-1)$, $\o_3^0(0,2,a-2)$, $\o_3^0(0,3,a-3)$, $\o_3^0(0,4,a-4)$, and $\o_3^0(0,5,a-5)$ as a vertex algebraic polynomial in terms of lower weight terms.

Next, for $5\leq a<b$, $D_5(b+2,a-1,0,0,0)$ can be used to construct
\be\begin{aligned}\label{step4}
\o_3^0(0,a&+1,b)=\frac{2}{a(a+1)(b+1)}((-1)^b-(-1)^a)\o_3^0(0,a-1,b+2)\\&+\binom{a+b+1}{b}\left(\frac{1+3(-1)^a}{a+1}+\frac{1-3(-1)^b}{b+2}+\frac{(-1)^b-(-1)^a}{a+b+1}\right)\o_3^0(0,0,a+b+1)\\&+\Psi,\end{aligned}\ee
where $\Psi$ is a vertex algebraic polynomial with terms of the form $\o_3^0(0,m_1,m_2)_{-1-n}\1$ and $\o_2^0(r_1,r_2)_{-1}\o_3^0(0,s_1,s_2)$ where $m_1+m_2+n=a+b+1$ and $r_2+r_2+s_1+s_2=a+b-1$ with $n\geq 1$. 

Finally, \eqref{step1} -- \eqref{step4}, provide us with a clear inductive path to eliminating all cubic generators other than $\o^0_{3}(0,0,0),\o^0_3(0,0,2),$ and $\o_3^0(0,1,2)$.
\color{black} This proves the first assertion. To prove the second claim we use Lemma \ref{quad-lem}, where we reduced all quadratic generators down to $\o_2^0(0,0),\o_2^0(0,2),\o_2^0(0,4)$. Although in this reduction 
we additionally used cubic generators $\o_3^0(0,0,1)$, $\o_3^0(0,1,1)$, and $\o_3^0(0,2,2)$, we can remove them using the first two relations in \eqref{setupderivative}, second relation in \eqref{bigderivative} and relation \eqref{step1}.
 \color{black}
\end{proof}

\begin{thm} \label{main}

(i) The vertex operator algebra $\H(3)^{S_3}$ is {simple} of type $(1,2,3,4,5,6^2)$, i.e. it is strongly generated 
by seven vectors whose conformal weights are: $1,2,3,4,5,6,6$. This generating set is minimal.
 
(ii) $\H(3)^{S_3}$ is isomorphic to $\H(1) \otimes W$, where $W$ is of type $(2,3,4,5,6^2)$.

(iii) $\H(3)^{S_3}$ is not freely generated (by any set of generators).

\end{thm}
\begin{proof}
Clearly, by a result from \cite{DM}, this vertex algebra is simple.
By Lemma \ref{cubic-lem} and earlier discussion we see that $\H(3)^{S_3}$ is strongly generated by 
vectors of conformal weights: $1$ ( linear generator), $2,4,6$ (quadratic generators) and $3,5,6$ (cubic generators).
From the character formula (see Proposition \ref{char-S3}) we see that 
$$q^{1/8} {\rm ch}[\H(3)^{S_3}](\tau)-\frac{1}{(q;q)_\infty (q^2;q)_\infty (q^3;q)_\infty (q^4;q)_\infty (q^5;q)_\infty (q^6;q)_\infty^2} = O(q^9),$$
where $(a;q)_\infty:=\prod_{i \geq 0} (1-aq^i)$.
An easy analysis shows that dropping one (or more) generators from this generating set would imply that certain graded dimensions of 
$\H(3)^{S_3}$ are strictly bigger than the corresponding graded dimension for the smaller subalgebra. Thus, the proposed set  of generators must be a minimal generating set. 
Clearly, this vertex algebra is not freely generated by these generators due to 
$$q^{1/8} {\rm ch}[\H(3)^{S_3}](\tau)-\frac{1-q^9}{(q;q)_\infty (q^2;q)_\infty (q^3;q)_\infty (q^4;q)_\infty (q^5;q)_\infty (q^6;q)_\infty^2} = O(q^{10}),$$
which implies that there must be a nontrivial relation at degree $9$.
In fact, one can quickly argue this is impossible simply from the fact that ${\rm ch}[\H(3)^{S_3}]$ is modular, which is impossible to achieve with a free generating set.

For (ii), we first observe that $\langle \omega_1(0) \rangle$ is isomorphic to $\H(1)$. By taking the commutant
$W:={\rm Comm}(\H(1), \H(3)^{S_3})$, we get  $\H(1) \otimes W \cong \H(3)^{S_3}$. Each generator of $ \H(3)^{S_3}$, except $\omega_1(0)$, can be written as a linear combination of elements 
in $\H(1) \otimes W$ with a nonzero component in $W$ (otherwise it would imply that some of the generators are contained in $\langle \omega_1(0) \rangle$, contradicting 
the minimality in (i)). These nonzero components form a generating set of $W$.

\end{proof}

Using (\ref{generatortranslation}) we may take the original invariants $\o_1(0), \o_2(0,0), \o_2(0,2)$, $\o_2(0,4)$, $\o_3(0,0,0)$,  $\o_3(0,0,2),$ and $\o_3(0,1,2)$ as our minimal strong generating set. Furthermore, the following change of variables allows us to express the orbifold in terms of primary generators \be\begin{aligned}\label{primary}
h&=\o_1(0), \ \  \omega=\frac{1}{2}\o_2(0,0), \ \  \ \ \ J_1=2\o_2(0,2)-\frac{24}{37}\o_{-1}\o-\frac{30}{37}\o_{-3}\1, \\
J_2&=24\o_2(0,4)-\frac{460}{81}\o_2(0,2)_{-3}\1-\frac{80}{27}\o_{-1}\o_2(0,2)+\frac{256}{801}\o_{-1}\o_{-1}\o\\
&-\frac{364}{2403}\o_{-2}\o_{-2}\1+\frac{3472}{2403}\o_{-3}\o+\frac{40744}{2403}\o_{-5}\1,\\
C_1&=\o_3(0,0,0)-\frac{1}{3}h_{-1}h_{-1}h,\\
C_2&=2\o_3(0,0,2)-\frac{2}{15}\o_3(0,0,0)_{-3}\1-\frac{8}{45}h_{-3}h_{-1}h+\frac{2}{15}h_{-2}h_{-2}h\\&-\frac{16}{45}\o_{-1}\o_3(0,0,0)
+\frac{16}{135}h_{-1}h_{-1}h_{-1}\o,\\
C_3&=2\o_3(0,1,2)+\frac{1}{10}\o_3(0,0,2)_{-2}\1-\frac{6}{25}\o_3(0,0,0)_{-4}-\frac{2}{25}\o_{-1}\o_3(0,0,0)_{-2}\1\\&+\frac{3}{25}\o_{-2}\o_3(0,0,0)-\frac{1}{25}h_{-1}h_{-1}h_{-1}\o_{-2}\1+\frac{2}{25}h_{-2}h_{-1}h_{-1}\o+\frac{2}{25}\o_{-2}\o_{-2}\o_{-2}\1\\&-\frac{4}{25}h_{-3}h_{-2}h+\frac{6}{75}h_{-4}h_{-1}h.
\end{aligned}\ee


\section{The orbifold $\H(3)^{\mathbb{Z}_3}$}

We now consider the orbifold of $\H(3)$ under the action of the cyclic subgroup $\mathbb{Z}_3\cong \left<\begin{pmatrix}1&2&3\end{pmatrix}\right>\subset S_3$. Using a similar strategy to the previous subsection we may take 
\be\begin{aligned}
 \o^0_1&=\b_1(-1)\1, \ \ \o^0_{2,3}(a,b)=\b_2(-1-a)\b_3(-1-b)\1,\\
 \o^0_{2,2,2}(a,b,c)&=\b_2(-1-a)\b_2(-1-b)\b_2(-1-c)\1, \\
   \o^0_{3,3,3}(a,b,c)& =\b_3(-1-a)\b_3(-1-b)\b_3(-1-c)\1
\end{aligned}\ee
as our initial generating set. Analogous to Lemma \ref{orbifoldgenerators1} we have the following initial reduction of the generating set
\begin{lem} The orbifold $\H(3)^{\mathbb{Z}_3}$ is strongly generated by vectors 
\be\label{cyclicred1}\begin{aligned}
\o_1^0(0)&, \ \ \o_{2,3}^0(0,a) \text{ for } a\geq 0,\\
\o_{2,2,2}^0(0,a,b)&  \text{ for } 0\leq a \leq b, \ \  \o_{3,3,3}^0(0,a,b)& \text{ for } 0\leq a \leq b.
\end{aligned}\ee
\end{lem}

\begin{proof}
The linear generator is the same vector found in Lemma \ref{orbifoldgenerators1} which may be used to produce all vectors of the form $\o_1^0(a)$ for $a\geq 0$ as before.

 The quadratic generators will be reduced using an argument similar to the proof of Lemma \ref{orbifoldgenerators1}, where the main difference is due to the fact that for $a\neq b$, $\o_{2,3}^0(a,b)\neq \o_{2,3}^0(b,a)$. We begin by setting 
\be
A(m)=\text{span}\{\o_{2,3}(a,b)|a+b=m\} \text{ and } \partial A(m)=\{v_{-2}\1 | v\in A(m)\}
\ee
and notice that for all $m\geq 0$, $\partial A(m)$ is a subspace of $A(m+1)$ of co-dimension 1. In fact, we have 
\be
A(m+1)=\partial A(m)\oplus \CC\o_{2,3}^0(0,m+1),\ee
from which it follows that the set 
\be
\{\o_{2,3}^0(0,i)_{-1-m+i}\1\}\ee
is a basis of $A(m)$. Our result follows.

Again, the reduction of the cubic generators follows from an analogue of \eqref{derivative}. In particular
\be \o_{i,i,i}^0(a,b,c)_{-2}\1=(a+1) \o_{i,i,i}^0(a+1,b,c)+(b+1) \o_{i,i,i}^0(a,b+1,c)+ (c+1)\o_{i,i,i}^0(a,b,c+1),\ee
for $i\in\{1,2\}$ may be used inductively to achieve the desired set, where we use the fact that $\o^0_{i,i,i}(a,b,c)$ is fixed under any permutation of the entries $a,b,c$.

\end{proof}

We now present analogues of Lemmas \ref{quad-lem} and \ref{cubic-lem} in this setting.

\begin{lem}
The full list of quadratic generators described in (\ref{cyclicred1}) can be replaced with 
\be
\o^0_{2,3}(0,0), \o^0_{2,3}(0,1),\o^0_{2,3}(0,2), \text{ and } \o^0_{2,3}(0,3).\ee\end{lem}
\begin{proof}
Our main tool for this argument will be the expression
\be\label{step1p}
D_2(a)=\o_{2,3}^0(0,a)_{-1}\o_{2,3}^0(1,1)-\o_{2,3}^0(1,a)_{-1}\o_{2,3}^0(0,1),\ee
\color{black}which may be expanded to 
\be\label{step2p}
D_2(a)=\binom{a+2}{2}\binom{a+4}{2}\o^0_{2,3}(0,a+4)+2\binom{a+3}{3}\o^0_{2,3}(1,a+3)-(-1)^a(a+1)\o^0_{2,3}(a+3,1).\ee
An elementary calculation allows us to write
\be\label{step3p}
\o^0_{2,3}(1,a+3)=-(a+4)\o^0_{2,3}(0,a+4)+\o_{2,3}^0(0,a+3)_{-2}\1,\ee
and more generally
\be\begin{aligned}\label{step4p}
\o^0_{2,3}(a+3,1)&=\sum_{j=0}^{a+3}(-1)^{a+j+1}(a-j+4)\o^0_{2,3}(0,a-j+4)_{-1-j}\1\\
&=(-1)^{a+1}(a+4)\o^0_{2,3}(0,a+4)\\&\hspace{.3in}+\sum_{j=1}^{a+3}(-1)^{a+j+1}(a-j+4)\o^0_{2,3}(0,a-j+4)_{-1-j}\1.
\end{aligned}\ee
Now combining \eqref{step2p}-\eqref{step4p} we have
\be\begin{aligned}
D_2(a)&=-\frac{1}{12}(a+6)(a+4)(a+1)(a-1)\o_{2,3}^0(0,a+4)\\&+2\binom{a+3}{3}\o_{2,3}^0(0,a+3)_{-2}\1\\&+\sum_{j=1}^{a+3}(-1)^j(a+1)(a-j+4)\o_{2,3}^0(0,a-j+4)_{-1-j}\1,
\end{aligned}\ee
which, for $a=0$ and $a\geq 2$, may be used to solve for $\o_{2,3}^0(0,a+4)$ in terms of vectors of lower conformal weight.
\color{black}

Inductively, together with the special case 
\be\begin{aligned}
\o_{2,3}^0(0,5)&=\frac{1}{7}(\o_{2,3}^0(0,0)_{-1}\o_{2,3}^0(1,2)-\o_{2,3}^0(1,0)_{-1}\o_{2,3}^0(0,2))\\&+\frac{3}{7}\o_{2,3}^0(0,4)_{-2}\1-\frac{3}{7}\o_{2,3}^0(0,3)_{-3}\1+\frac{1}{7}\o_{2,3}^0(0,2)_{-4}\1,\end{aligned}\ee
this allows us to remove all but the necessary generators.
\end{proof}


\begin{lem}
The full list of cubic generators described in (\ref{cyclicred1}) can be replaced with
\be\label{cycliccubicfinal}
\o_{2,2,2}^0(0,0,0), \o^0_{2,2,2}(0,0,2),\o_{3,3,3}^0(0,0,0), \o^0_{3,3,3}(0,0,2).\ee
\end{lem}

\begin{proof}
We follow a similar strategy to the proof of Lemma \ref{cubic-lem}. In this case, our argument is greatly simplified due to the fact that our decoupling relations are simpler and occur at a lower initial conformal weight. We focus on the $\o_{2,2,2}^0(0,a,b)$ terms, as the $\o_{3,3,3}^0(0,a,b)$ are similar, starting with the observation that 
\be\begin{aligned}
\o_{2,2,2}^0(0,0,1)&=\frac{1}{3}\o_{2,2,2}^0(0,0,0)_{-2}\1\\
\o_{2,2,2}^0(0,1,1)&=-\o_{2,2,2}^0(0,0,2)+\frac{2}{3}\o_{2,2,2}^0(0,0,0)_{-3}\1,\end{aligned}\ee
meaning that, by using the translation operator, all cubic generators of weight five or less may be written in terms of \eqref{cycliccubicfinal}.

 The expression 
\be\label{z3genrels}
D_3(a,b)=\o_{2,3}^0(0,a)_{-1}\o^0_{2,2,2}(0,0,b)-\o_{2,3}^0(b,a)_{-1}\o^0_{2,2,2}(0,0,0)\ee
will be our main tool moving forward. A special case of \eqref{z3genrels} can be used to construct the equation
\be
\o_{2,2,2}^0(0,0,a)=\frac{1}{3a+1}\left(\o_{2,2,2}^0(0,0,a-1)_{-2}\1+\frac{(-1)^a}{a-2}D_3(a-3,1)\right).\ee
which can be used inductively to remove all generators of the form $\o_{2,2,2}^3(0,0,a)$ for $a\geq 3$ from the generating set. Furthermore, a more general version yields 
\be
\o_{2,2,2}^0(0,a,b)=\frac{1}{2}\left(\frac{2a+3b}{a+b}\binom{a+b}{a}\o_{2,2,2}^0(0,0,a+b)+\frac{(-1)^a}{a-1}D(a-2,b)\right)\ee
which can be used to remove all generators of the form $\o_{2,2,2}^3(0,a,b)$ where $2\leq a\leq b$. This leaves only the generators of the form $\o_{2,2,2}^0(0,1,a)$ which can be removed using 
\be
\o_{2,2,2}^0(0,1,a)=\frac{1}{2}(\o_{2,2,2}^0(0,0,a)_{-2}\1-(a+1)\o_{2,2,2}^0(0,0,a+1)).\ee

\end{proof}

\begin{thm} \label{cyclicmain}

(i) The vertex operator algebra $\H(3)^{\mathbb{Z}_3}$ is {simple} of type $(1,2,3^3,4,5^3)$, i.e. it is strongly generated 
by seven vectors whose conformal weights are: $1,2,3,3,3,4,5,5,5$. This generating set is minimal.
 
(ii) $\H(3)^{\mathbb{Z}_3}$ is isomorphic to $\H(1) \otimes W$, where $W$ is of type $(2,3^3,4,5^3)$.

(iii) $\H(3)^{\mathbb{Z}_3}$ is not freely generated (by any set of generators).

\end{thm}

\begin{proof}
This proof is similar to the proof of Theorem \ref{main}. In particular, we use the fact that the free $\mathcal{W}$-algebra on a generating set containing the same weight elements with any one removed has certain graded dimensions strictly less than those described by the character of $\H(3)^{\mathbb{Z}_3}$, (\ref{z3-char}).
\end{proof}

Finally, the orbifold $\H(3)^{\mathbb{Z}_3}$ is described in terms of primary generators with the following vectors.

\be\begin{aligned}
h&=\o_{1}^0(0), \ \ \o=\o_2^0(0,0)+\frac{1}{2}\o_1^0(0)_{-1}\o_1^0(0), \ \ \ \ \ J_1=\o_2^0(0,1)-\frac{1}{2}\o_{-2}\1+\frac{1}{2}h_{-2}h,\\
J_2&=2\o_2^0(0,2)-\o_2^0(0,1)_{-2}\1+\frac{2}{3}\o_{-1}\o-\frac{22}{3}h_{-1}h_{-1}\o+\frac{77}{9}h_{-3}h\\&-\frac{83}{12}h_{-2}h_{-2}\1,\\
J_3&=6\o_2^0(0,3)-\frac{8}{3}\o_2^0(0,2)_{-2}\1+4h_{-2}h_{-1}\o+8h_{-1}h_{-1}\o_{-2}\1\\&-16h_{-3}\o+8h_{-1}\o_{-3}\1+4h_{-1}^4h-\frac{34}{3}h_{-2}h_{-2}h-\frac{40}{3}h_{-3}h_{-1}h+4h_{-1}^4h\\&+\frac{58}{3}h_{-3}h_{-2}\1-23h_{-4}h+\frac{88}{5}h_{-5}\1,\\
C_1^{(2)}&=\o^0_{2,2,2}(0,0,0), \ \ C_2^{(2)}=\o^0_{2,2,2}(0,0,2)-\frac{1}{15}\o_{2,2,2}^0(0,0,0)_{-3}\1-\frac{8}{15}\o_{-1}\o_{2,2,2}^0(0,0,0),\\
C_1^{(3)}&=\o^0_{3,3,3}(0,0,0), \ \ C_2^{(3)}=\o^0_{3,3,3}(0,0,2)-\frac{1}{15}\o_{3,3,3}^0(0,0,0)_{-3}\1-\frac{8}{15}\o_{-1}\o_{3,3,3}^0(0,0,0).
\end{aligned}\ee

\section{The character of $\H(n)^{S_n}$}

We first derive a general formula for the character of the $S_n$-orbifold of a vertex algebra. Although formulas of this form have appeared in the literature (e.g., \cite{B} or \cite{KT}), we include the proof for 
completeness.

\begin{thm} Let $V$ be a VOA and $f^{S_n}(q)$ the character of the $S_n$-fixed point subalgebra $V(n)^{S_n} \subset V^{\otimes^n}$ (under the usual action). Then 
$$f^{S_n}(q)=\frac{1}{n!} \sum_{g \in S_n} {\rm tr}_{V^{\otimes n}} g q^{L(0)-c/24}$$
where $g$ acts by permutation of tensor factors.
\end{thm}
\begin{proof} Consider the group algebra $\mathbb{C}[S_n]$ and the idempotent 
$$e=\frac{1}{|S_n|} \sum_{g \in S_n} g \in \mathbb{C}[S_n], \ \ \ e^2=e$$
where $|S_n|=n!$. The space ${\rm Im}(e)$ is precisely $V(n)^{S_n}$. Since the eigenvaues of $e$ are $1$ and $0$ (and $1$ for the $S_n$-fixed 
subalgebra), the character 
can be computed simply by taking the trace of $e$:
$${\rm tr}_{V^{\otimes^n}} e q^{L(0)-c/24}={\rm tr}_{V(n)^{S_n}} q^{L(0)-c/24}.$$
\end{proof}

Now we specialize to $n=3$ and $V=\H(1)$.

\begin{prop} \label{char-S3}
$${\rm ch}[\H(3)^{S_3}](\tau)=\frac{q^{-1/8}}{6}\left(\prod_{n=1}^\infty \frac{1}{(1-q^n)^3}+ 3 \prod_{n=1}^\infty \frac{1}{(1-q^{2n})} \prod_{n=1}^\infty \frac{1}{(1-q^n)}
+2 \prod_{n=1}^\infty \frac{1}{(1-q^{3n})}\right).$$
\end{prop}

\begin{proof} It suffices to consider elements: $g=1$, $g=(12)$ and $g=(123)$. By the previous theorem we have 
$${\rm ch}[\H(3)^{S_3}](\tau)=\frac{1}{6} \left( {\rm tr}_{\H(3)} q^{L(0)-3/24}+{3} \cdot {\rm tr}_{\H(3)} (12) q^{L(0)-3/24}+{2} \cdot {\rm tr}_{\H(3)} (123) q^{L(0)-3/24}\right)$$
Clearly, 
$${\rm tr}_{\H(3)} q^{L(0)-3/24}=q^{-1/8} \prod_{n=1}^\infty \frac{1}{(1-q^n)^3}.$$
For ${\rm tr} (12)$ 
we compute the trace on the following basis of $\H(3)$ 
$$\mathcal{B}=\{ p_{\lambda_1} \otimes p_{\lambda_2} \otimes p_{\lambda_3}, \lambda_i \in \mathcal{P} \}$$
where $\mathcal{P}$ is the set of partitions, and for $p_{\lambda_i}$ we take obvious monomials of Heisenberg elements.
The matrix representation of $(12) \in {\rm End}(\H(3))$ in this basis has a non-zero entry on the diagonal (=1) if and only if the corresponding basis element is 
$$p_{\lambda} \otimes p_{\lambda} \otimes p_{\nu}.$$
Since all $p_{\lambda}$ are generated by $\alpha(-i)$, vectors contributing to the trace are generated by $\alpha(-i) \otimes \alpha(-i) \otimes \alpha(-j)$. 
Thus 
$${\rm tr}_{\H(3)} (12) q^{L(0)-3/24}=q^{-1/8} \prod_{i=1}^\infty \frac{1}{(1-q^{2i})} \prod_{j=1}^\infty \frac{1}{(1-q^j)}.$$
One similarly argues for $3$-cycles, so we get 
$${\rm tr}_{\H(3)} (123) q^{L(0)-3/24}=q^{-1/8} \prod_{i=1}^\infty \frac{1}{(1-q^{3i})}.$$
\end{proof}


\begin{cor}\label{z3-char}
$${\rm ch}[\H(3)^{\mathbb{Z}_3}](\tau)=\frac{q^{-1/8}}{3}\left(\prod_{n=1}^\infty \frac{1}{(1-q^n)^3}
+2 \prod_{n=1}^\infty \frac{1}{(1-q^{3n})}\right).$$
\end{cor}








\section{ Modular invariance of $\H(3)^{S_3}$-characters and quantum dimensions}


By a result from \cite{DM} (especially Section 7 devoted to dihedral groups), we have a decomposition of $\H(3)$ in terms of $\H(3)^{S_3}$-modules:
\begin{equation} \label{decomp}
\H(3)=\H(3)^{S_3} \oplus  \H(3)^{S_3,sgn} \oplus V(2) \otimes \H(3)^{S_3,st}
\end{equation}
where $V(2)$ is the $2$-dimensional standard representation of $S_3$, and $\H(3)^{S_3,sgn}$ and $\H(3)^{S_3,st}$ are irreducible 
$\H(3)^{S_3}$-modules. 

\begin{lem} We have
\begin{align*}
{\rm ch}[\H(3)^{S_3,sgn}](\tau)&=\frac{q^{-1/8}}{6}\left(\prod_{n=1}^\infty \frac{1}{(1-q^n)^3}- 3 \prod_{n=1}^\infty \frac{1}{(1-q^{2n})} \prod_{n=1}^\infty \frac{1}{(1-q^n)}
+2 \prod_{n=1}^\infty \frac{1}{(1-q^{3n})}\right) \\
{\rm ch}[\H(3)^{S_3,st}](\tau)&=\frac{q^{-1/8}}{6}\left(2 \prod_{n=1}^\infty \frac{1}{(1-q^n)^3}- 2 \prod_{n=1}^\infty \frac{1}{(1-q^{3n})}\right) .
\end{align*}
\end{lem}
\begin{proof} 
The character of the sign module \cite{DM}
$$\H(3)^{S_3,sgn}=\{ a \in \H(3) : \sigma(a)={\rm sgn}(\sigma) a, \sigma \in S_3 \}$$
is computed along the lines of Proposition \ref{char-S3} with the difference that in the sign representation $2$-cycles act as $-1$. This sign contributes with a negative sign to the middle terms of the formula. 

The second relation follows  directly from the first formula, Proposition \ref{char-S3},  and decomposition (\ref{decomp}).
\end{proof}

Next we construct several $\H(3)^{S_3}$-modules coming from ordinary and $g$-twisted $\H(3)$-modules, $g \in S_3$. 

We first look at $\H(3)^{S_3}$-modules coming from irreducible $\H(3)$-modules (Fock representations).
Let ${\bf w}=(w_1,w_2,w_3) \in \mathbb{C}^3$. We denote by $F_{w_1,w_2,w_3}$ the irreducible $\H(3)$-module 
with highest weight ${\bf w}$ such that $\alpha_i(0)$ acts as multiplication by $w_i$. Similarly, $F_{w}$ denotes the Fock representation for $\H(1)$ with highest weight $w$.

Interestingly, $F_{w_1,w_2,w_3}$ is generically irreducible as an $\H(3)^{S_3}$-module.
\begin{lem} \label{typical}
Let $w_1,w_2,w_3 \in \mathbb{C}^3$ such that $w_i \neq w_j$, $i \neq j$. Then the  $\H(3)$-module $F_{w_1,w_2,w_3}$  
is irreducible as an $\H(3)^{S_3}$-module.

\end{lem}
\begin{proof} Let $V$ be a simple VOA, $g$ be an automorphism of $V$ of prime order $p$, and $M$ be a simple
$V$-module such that $ g \circ M$  is not isomorphic to $M$ as $V$-modules. Then according to Theorem 6.1 \cite{DM}, $M$ is a simple $V^{<g>}$-module.
To apply this result we first observe embeddings 
$$\H(3)^{S_3} \hookrightarrow \H(3)^{<(12)>}  \hookrightarrow \H(3).$$
Because of $w_1 \neq w_2$,  $(12) \circ F_{w_1,w_2,w_3} \ncong F_{w_1,w_2,w_3}$ as $\H(3)$-modules. Indeed, for $(12) \circ F_{w_1,w_2,w_3}$ 
actions of $\alpha_1(0)$ and $\alpha_2(0)$ are switched, so $(12) \circ F_{w_1,w_2,w_3} \cong F_{w_2,w_1,w_3}$.
Therefore  $F_{w_1,w_2,w_3}$ is irreducible as a $G(3):=\H(3)^{(<12>)}$-module.
Similarly, we consider $(123) \circ F_{w_1,w_2,w_3}$ and $F_{w_1,w_2,w_3}$ as $G(3)$-modules.
Because of $w_1 \neq w_3$ and $w_3 \neq w_2$, they are not isomorphic as $G(3)$-modules and thus $F_{w_1,w_2,w_3}$ is irreducible as a $G(3)^{<(123)>}$-module. The proof follows.
\end{proof}

In the study of characters we require modular variables so we let $q=e^{2 \pi i \tau}$, where $\tau \in \mathbb{H}$ (the upper half-plane).
Let 
$$\eta(\tau)=q^{1/24} (q;q)_\infty,$$
the Dedekind $\eta$-function.
Clearly, for every ${\bf w}$,
\begin{equation} \label{generic}
{\rm ch}[F_{w_1,w_2,w_3}](\tau)=\frac{q^{-3/24 + w_1^2/2+w_2^2/2+w_3^2/2}}{(q;q)_\infty}=\frac{q^{w_1^2/2+w_2^2/2+w_3^2/2}}{\eta(\tau)^3}.
\end{equation}

Next we consider $g$-twisted $\H(3)$-modules, so $g$ is either a $2-$ or $3$-cycle. We may assume, without loss of generality, that $g$ is either $(12)$ or $(123)$.

We first construct a family of $\theta$-twisted $\H(3)$-modules, where  $\theta:=(12)$ switches the first two "coordinates". By using the isomorphism in Section 3, $\H(2)^{<(12)>} \cong \H(1) \otimes M(1)^+$, for every $w_1,w_3 \in \mathbb{C}$, we get
a $\theta$-twisted $\H(3)$-module
$$F_{w_1,w_3}(\theta):=F_{w_1} \otimes M(1)(\theta) \otimes F_{w_3},$$
 where $M(1)(\theta) \cong \mathbb{C}[\alpha(-1/2),\alpha(-3/2),....]$ is a $\mathbb{Z}_2$-twisted $\H(1)$-module studied in \cite{DN}. 
This module is of conformal weight $\frac{1}{16}$ and its character is easily computed. We quickly  get 
$${\rm ch}[F_{w_1,w_3}(\theta)](\tau)=\frac{q^{1/16-3/24 +{w_1^2}/{2}+{w_3^2}/{2}}}{(q^{1/2};q)_\infty (q;q)_\infty (q;q)_\infty}=\frac{q^{-1/16 + {w_1^2}/{2}+{w_3^2}/{2}}}{(q^{1/2};q^{1/2})_\infty (q;q)_\infty}$$
\begin{equation} \label{12-char}
=\frac{q^{{w_1^2}/{2}+{w_3^2}/{2}}}{\eta(\tau/2) \eta(\tau)}.
\end{equation}

Further, we construct a family of $(123)$-twisted $\H(3)$-modules. Let $\sigma=(123)$.
We first construct a $\sigma$-twisted Heisenberg algebra $\hat{\frak{h}}^\sigma$ as in \cite{Le}. It is isomorphic to  
$\mathbb{C}[t,t^{-1}] \oplus t^{1/3}\mathbb{C}[t,t^{-1}] \oplus t^{2/3}\mathbb{C}[t,t^{-1}] \oplus \mathbb{C}k$ as a graded Lie algebra. Then for every $w \in \mathbb{C}$ we have an associated $\sigma$-twisted $\H(3)$-representation $F_w(\sigma) \cong U(\frak{h}^\sigma_{<0})$. Moreover,
$${{\rm ch}}[F_w(\sigma)](\tau)=\frac{q^{w^2/2+h-1/8}}{(q;q)_\infty (q^{1/3};q)_\infty (q^{2/3};q)_\infty}=\frac{q^{w^2/2+h-3/24}}{(q^{1/3};q^{1/3})_\infty},$$
where $h=\frac{1}{4p^2} \sum_{i=1}^{p-1} i(p-i)r_i$,
and $p$ is the order of $\sigma$ and $r_i$ ($1 \leq i \leq 3$) are dimensions of the eigenspaces of $\sigma$. Plugging in $p=3$ and $r_i=1$ into the last formula we obtain $h=\frac{1}{9}$ and thus
\begin{equation} \label{123}
{\rm ch}[F_w(\sigma)](\tau)=\frac{q^{w^2/2-1/72}}{(q^{1/3};q^{1/3})_\infty}=\frac{q^{w^2/2}}{\eta(\tau/3)}.
\end{equation}

Here we do not pursue decomposition and irreducibility of $F_{w_1,w_3}(\sigma)$ and $F_w(\sigma)$ as $\H(3)^{S_3}$-modules. This will be a subject of \cite{MP2}.

\subsection{Modular invariance}
In this part we obtain a modularity result for the character of $\H(3)^{S_3}$ under the $S$-transformation, $\tau \to -\frac{1}{\tau}$. This was briefly discussed in the introduction. 

We make use of a well-known  modular relation for the  Dedekind $\eta$-function:
\begin{equation} \label{dedekind}
\eta(-1/\tau)^n=(\sqrt{-  i \tau })^n \eta(\tau)^n, \ \ n \in \mathbb{N}.
\end{equation}

Let us also recall a higher dimensional Gauss' integral formula
$$\int_{\mathbb{R}^n} q^{w_1^2/2+\cdots + w_n^2/2} d {\bf w}=\frac{1}{(\sqrt{-i \tau})^n}.$$
From the Gauss formula and (\ref{dedekind}) we immediately get the following relations.
\begin{lem} \label{modularity}
\begin{align*}
\frac{1}{\eta(-1/\tau)^3}&=\int_{\mathbb{R}^3} \frac{q^{w_1^2/2+w_2^2/2+w_3^2/2}}{\eta(\tau)^3} d{\bf w}, \ {\bf w}=(w_1,w_2,w_3) \\
 \frac{1}{\eta(-1/\tau) \eta(-2/\tau)}& =\sqrt{2} \int_{\mathbb{R}^2}  \frac{q^{w_1^2/2+w_3^2/2}}{\eta(\tau)\eta(\tau/2)} d{\bf w} , \  \ \ \  {\bf w}=(w_1,w_3)\\
\frac{1}{\eta(-3/\tau)}&=\sqrt{3} \int_{\mathbb{R}} \frac{q^{w^2/2}}{\eta(\tau/3)} dw.
\end{align*}
\end{lem}
Notice that only the numerator of the integrand depends on ${\bf w}$. The denominator is placed inside the integral so that it matches the shape of (\ref{gen-mod}).

Proposition \ref{char-S3}, the character formulas for $F_{w_1,w_2,w_3}$, $F_{w_1,w_3} (\theta)$ and $F_{w}(\sigma)$, that is,  formulas (\ref{generic}), (\ref{12-char}) and (\ref{123}), together with Lemma \ref{modularity} imply 
\begin{thm}
The character of $\H(3)^{S_3}$ has a modular invariance property, in the sense of (\ref{gen-mod}).
\end{thm}


\subsection{Quantum dimensions}
Let $V$ be a vertex operator algebra.
For a $V$-module $M$ we define its {\em quantum dimension}:
$${\rm qdim}[M]:=\lim_{t \to 0^+} \frac{{\rm ch}[M](it)}{{\rm ch}[V](it)}.$$
In a situation where $V$ is rational (in the strongest sense) this limit always exists and is nonzero (see \cite{DRX1}). But for general $V$ this limit can be zero or even $+\infty$.

As our characters can be expressed in terms of $\eta(\tau)$, we will make use of  an asymptotic formula
\begin{equation} \label{asym} 
\eta(it)^n \sim \left(\frac{1}{t^{\frac{n}{2}}}\right) e^{-\frac{\pi n}{12 t}}, \ \ (t \to 0^+).
\end{equation}

\begin{prop} \label{quantum} As $\H(3)^{S_3}$-modules, we have: 

(a) For any Fock representation $F_{w_1,w_2,w_3}$,
$${\rm qdim}[F_{w_1,w_2,w_3}]=6.$$

(b) For representations appearing inside  $\H(3)$:
\begin{align*}
& {\rm qdim}[\H(3)^{S_3}]=1 \\
& {\rm qdim}[\H(3)^{S_3,sgn}]=1 \\
& {\rm qdim}[\H(3)^{S_3,st}]=2.
\end{align*} 

(c)  $F_{w_1,w_3}(\theta)$ and $F_{w}(\sigma)$ have quantum dimension $+\infty$. 
\end{prop}
\begin{proof}
We first observe that in the formulas for ${\rm ch}[\H(3)^{S_3}](\tau)$, ${\rm ch}[\H(3)^{S_3,sgn}](\tau)$ and ${\rm ch}[\H(3)^{S_3,st}](\tau)$
the infinite product $\frac{q^{-1/8}}{(q;q)_\infty^3}=\frac{1}{\eta(\tau)^3}$, multiplied with a constant, dominates the asymptotics $t \to 0^+$. 
Therefore in order to compute quantum dimensions for these modules, we only have to compute the ratio of these constants. This implies assertions  (a) and (b).

For (c), we need more precise asymptotic behaviors. We have
\begin{align*}
& \frac{1}{\eta(\tau)\eta(\tau/2)}  \sim \frac{t}{\sqrt{2}}  e^{\frac{\pi}{4t}},  \\ 
& \frac{1}{\eta(\tau/3)}  \sim \frac{t^{1/2}}{\sqrt{3}} e^{ \frac{\pi}{4 t}}, \\
& \frac{1}{\eta(\tau)^3} \sim t^{3/2} e^{\frac{\pi}{4 t}}.
\end{align*}
From these formulas  we get
$$\frac{{\rm ch}[F_{w_1,w_3}(\theta)(it)}{{\rm ch}[\H(3)^{S_3}](it)}  \sim \frac{1}{\sqrt{2} t^{1/2}},$$
$$\frac{{\rm ch}[F_{w}(\sigma)(it)}{{\rm ch}[\H(3)^{S_3}](it)}  \sim \frac{1}{\sqrt{3} t},$$
so both ratios have  growing terms at $t=0$. The proof follows.

\end{proof}

We finish with a conjecture which will be addressed in \cite{MP2}.
\begin{conj} \label{twisted} Every irreducible (ordinary) $\H(3)^{S_3}$-module $M$ appears in the decomposition of a  $g$-twisted $\H(3)$-module, for $g \in S_3$. Moreover, 
$${\rm qdim}(M) \in \{1, 2,6, +\infty \}.$$
\end{conj}


\section{Future work}

\begin{rem}

\begin{enumerate}

\item  In order to confirm Conjecture \ref{twisted}, we first have to describe the Zhu algebra of $\H(3)^{S_3}$. 
This will be addressed in \cite{MP2}.

\item In our very recent work \cite{MP}, we determine strong minimal generating sets for the permutation orbifold ${\mathcal F}(3)^{S_3}$, where $\F$ is the free  fermion vertex algebra and for $\mathcal{SF}(3)^{S_3}$,  
where $\mathcal{SF}$ is the rank one symplectic fermion vertex superalgebra \cite{CL}. We prove that these vertex algebras are of type $\frac12,2,4,\frac92$ and $1^2,2,3^3,4^3,5^5,6^4$, respectively.

\end{enumerate}

\end{rem}

{\bf Acknowledgements:} 
We thank Drazen Adamovic, Thomas Creutzig and especially Andrew Linshaw for useful discussions. We also thank the referee for constructive criticism and other valuable 
comments.  Several computations in this paper are performed by the OPE package, \cite{T}, for Mathematica.

The first author was partially supported by the NSF grant DMS-1601070.

\vspace{.3in}

\vspace{.2in}

\end{document}